\documentclass[12pt]{amsart}

\newtheorem{theorem}{Theorem}[section]
\newtheorem{lemma}[theorem]{Lemma}
\newtheorem{corollary}[theorem]{Corollary}
\newtheorem{proposition}[theorem]{Proposition}
\newtheorem{remark}[theorem]{Remark}
\newtheorem{definition}[theorem]{Definition}

\newcommand{\ncom}{\newcommand}
\ncom{\lrar}{\longrightarrow}
\ncom{\ov}{\overline}
\ncom{\m}{\mbox}
\ncom{\vepsilon}{\varepsilon}
\ncom{\sta}{\stackrel}
\ncom{\comx}{{\mathbb C}}
\ncom{\Z}{{\mathbb Z}}
\ncom{\Q}{{\mathbb Q}}
\ncom{\R}{{\mathbb R}}
\ncom{\G}{{\mathbb G}}
\ncom{\al}{\alpha}
\ncom{\p}{{\mathbb P}}
\ncom{\E}{{\mathbb E}}
\ncom{\N}{{\mathbb N}}
\ncom{\K}{{\mathbb K}}
\ncom{\f}{\frac}
\ncom{\cA}{{\mathcal A}}
\ncom{\cD}{{\mathcal D}}
\ncom{\cX}{{\mathcal X}}
\ncom{\cO}{{\mathcal O}}
\ncom{\cW}{{\mathcal W}}
\ncom{\cL}{{\mathcal L}}
\ncom{\cP}{{\mathcal P}}
\ncom{\cH}{{\mathcal H}}
\ncom{\cS}{{\mathcal S}}
\ncom{\cM}{{\mathcal M}}
\ncom{\cC}{{\mathcal C}}
\ncom{\cT}{{\mathcal T}}
\ncom{\cF}{{\mathcal F}}
\ncom{\cN}{{\mathcal N}}
\ncom{\cJ}{{\mathcal J}}
\ncom{\cV}{{\mathcal V}}
\ncom{\cZ}{{\mathcal Z}}
\ncom{\cU}{{\mathcal U}}
\ncom{\cSU}{{\mathcal S \mathcal U}}
\ncom{\cG}{{\mathcal G}}
\ncom{\cQ}{{\mathcal Q}}
\ncom{\cR}{{\mathcal R}}
\ncom{\cI}{{\mathcal I}}

\begin{document}
\baselineskip=16pt

\title{Bundles of Verlinde spaces and group actions}

\author[J. N. Iyer]{Jaya NN Iyer}

\address{The Institute of Mathematical Sciences, CIT
Campus, Taramani, Chennai 600113, India}

\email{jniyer@imsc.res.in}
\footnotetext{Mathematics Classification Number: 14C25, 14D05, 14D20, 14D21 }
\begin{abstract}
A \textit{Verlinde space} of level $k$ is the space of global sections of the $k$-th power of 
the determinant line bundle on the moduli space $\cSU_C(r)$ of semi-stable bundles of rank $r$ on a
curve $C$.
The aim of this note is to make accessible some remarks on the action of the \textit{Theta group} on the Verlinde spaces of higher level. This gives a decomposition of the bundle of Verlinde spaces
over the moduli space of curves and we indicate how to compute the rank of the isotypical components in the decomposition. 

\end{abstract}

\maketitle

$$\textsc{Contents}$$

1. Introduction

2. The space $H^0(\cSU_C(r),\Theta_C)$ is a Heisenberg module 

3. Parabolic case

4. A decomposition of the Verlinde bundles of higher level

5. A remark on the multiplicities of the isotypical components

6. References

\section{Introduction}

Let $C$ be a nonsingular connected projective curve defined over $\comx$.
The Jacobian variety $J(C)$ associated to the curve is a moduli space of rank
one and degree zero bundles on the curve $C$. There is a natural polarization $\Theta_C$
on the Jacobian and one can associate the space $H^0(J(C), \Theta^k_C)$ of global sections of the $k$-th power of the line bundle $\Theta_C$, also called as the abelian theta functions. 
The \textit{Theta group} $\cG(\Theta_C^k)$ was introduced by Mumford \cite{Mu1} and he prescribed an action of this group on $H^0(J(C),\Theta^k_C)$ (more generally for sections of line bundles on abelian varieties, see \S 2.1). As an application, he obtained results
on equations defining abelian varieties amongst many other moduli questions.

A higher rank analogue of $J(C)$ is the moduli space $\cU_C(r,0)$ of semi--stable bundles of rank $r$ and degree $0$ and the moduli space $\cSU_C(r)$ of semi--stable vector bundles of rank $r$ and trivial determinant on $C$, introduced by Mumford, Narasimhan and Seshadri \cite{Mu}, \cite{Na-Se}, \cite{Se}. There is a polarization $\Theta$ on the moduli space $\cSU_C(r)$ called as the 
determinant bundle \cite{Dr-Na}.
The space $H^0(\cSU_C(r),\Theta^k)$ of global sections of $\Theta^k$ are called as the \textit{Verlinde spaces} of level $k$. The sections are  called  the generalized theta functions. An action of a theta group $\cG$ on the space 
$H^0(\cSU_C,\Theta)$ was prescribed in \cite{BNR} and it was shown to be an irreducible $\cG$--module. We wish to investigate the $\cG$--action on the higher 
level Verlinde spaces. 

We put this in the framework of families of these moduli spaces over the moduli space of curves. This is done to be able to compute the Chern classes of the bundle of the Verlinde spaces of level one and we hope that it finds applications on further questions on monodromy of the projective representations.

Suppose $\pi_C:\cC\lrar T$ is a smooth projective family of curves of genus $g$.
We can associate to this family, the relative moduli space
\begin{equation}\label{eq.-v}
\pi_S:\cSU_\cC(r)\lrar T
\end{equation}
of semi--stable vector bundles of rank $r$ and trivial determinant.
There is a relative polarization $\Theta$ on $\cSU_\cC(r)$, also called as the 
determinant bundle.  

The \textit{Verlinde bundles} $$\cV_{r,k}:=\,\pi_{S\,*}(\Theta^k)$$ are known to be equipped with a projectively flat connection (i.e., a flat connction on the projectivization $\p(\cV_{r,k})$), also called as Hitchin's connection (see \cite{Fa1}, \cite{Hi}).
We notice that $\Theta_C$ is not uniquely defined since we can tensor it by the pullback of any line bundle on $T$. This implies that the Verlinde bundles are
defined upto taking tensor product with a line bundle on $T$.

Let $\gamma_{r,k}=\m{dim}H^0(\cSU_{\cC_t}(r),\Theta^k_t)$ be the dimension of the space of sections
of $\Theta^k_t$. Then, by \cite{Be-La}, \cite{Fa2} we have the `Verlinde formula':

$$\gamma_{r,k}\,=\,(\f{r}{r+k})^g.\sum_{\sta{S\sqcup R=[1,r+k]}{|S|=r}}\prod_{\sta{s\in S}{z\in R}}|2.\m{sin}\,\pi\f{s-z}{r+k}|^{g-1}.$$

We show that there is a decomposition of the Verlinde bundle, of the form 
$$\bigoplus_{\chi\in \widehat{K(\delta)_{k}}}W_\chi\otimes F_\chi$$
over a suitable cover of $T$.
Here $W_\chi$ is an irreducible Heisenberg representation (of higher weight) and $F_\chi$ is a vector bundle on an \'etale cover of $T$ over any point (Proposition \ref{pr.-1}). This is an application of Mumford's Theorem \cite[Proposition 2, p.80]{Mu2} of theta groups, to the case of generalized theta functions.

 We indicate how the rank of the bundles $F_\chi$ can be computed (section $5$). This shows that the dimension of the isotypical components are different and the isotypical component corresponding to the trivial character is greater than the other components. This is in contrast with the abelian theta functions, where all the components are equi-dimensional (see \cite[Proposition 3.7]{Iy1}, which is stated for level $2$, but in fact it holds for any level).
As an application, we compute the Chern character of the level one Verlinde bundle in the rational Chow groups (Corollary \ref{co.-2}).

The proof is via a study of the Heisenberg group representations   \cite{Mu1}, \cite{Iy1}. 
We extend the action of the Heisenberg group to higher level Verlinde spaces to obtain our assertion.  
The action is prescribed in a more general set-up, i.e., for moduli of 
parabolic bundles. Our hope was to compute the multiplicities using 
degeneration of the moduli spaces with their polarizations and using the 
\textit{Factorisation theorems}. It then becomes essential to consider moduli of parabolic bundles with a $\cG$-action on the space of generalized theta sections . The Factorization theorems were proved by Faltings, Narasimhan, Ramadas, Sun  \cite{Fa2}, \cite{Na-Ra}, \cite{Su} and many other mathematicians in computing the Verlinde formula in some 
cases. It seemed difficult for us to carry out the computations with a $\cG$--action though. We include \S 3 for the interested 
readers who might want to use this approach.
 Beauville, Laszlo, Sorger \cite{Be1}, \cite{Be-La-So}, Andersen-Masbaum \cite{An-Ma} have treated special cases.
  
{\Small Acknowledgements:  We thank H. Esnault for bringing our attention to the the Verlinde bundles, in summer 2004.  
 We thank G. Masbaum for his interest on the contents and encouraging us to pursue it further.}

\section{The space $H^0(\cSU_C(r),\Theta_C)$ is a Heisenberg module}

All the varieties are considered over the field of complex numbers.

\subsection{Theta groups}
We recall the definition of the Theta group introduced by Mumford and refer to 
\cite{Mu1} for details.

Suppose $A$ is an abelian variety of dimension $g$ and let $L$ be an ample line bundle on $A$. Consider the translation map, for any $a\in A$ :
 $$t_a:A\lrar A,\,x\mapsto x+a.$$
Consider the group : 
$$K(L) \,= \, \{a\in A: L\simeq t^*_aL\}$$ 
and the Theta group of $L$:

$$\cG(L) \,=\, \{(a,\phi): L\sta{\phi}{\simeq }t^*_aL\}.$$

In particular there is a central extension :
$$1\lrar \comx^*\lrar \cG(L)\lrar K(L)\lrar 0.$$

\subsection{Heisenberg groups}

Fix positive integers $\delta_1,\delta_2,...,\delta_g$ such that $\delta_i$ divides $\delta_{i+1}$, for each $i$. The $g$--tuple $\delta=(\delta_1,\ldots,\delta_g)$ is called the $type$ of $\delta$.

Given a type $\delta$, write
\begin{eqnarray*}
K_1(\delta)& = & (\f{\Z}{\delta_1\Z}\times...\times \f{\Z}{\delta_g\Z}) \\
\widehat{K_1(\delta)} & = & \m{Group of characters on }K_1(\delta) \\
K(\delta) &=& K_1(\delta) \oplus \widehat{K_1(\delta)}.\\
\end{eqnarray*}

The Heisenberg group $Heis(\delta)$ is the set
$$\comx^*\times K(\delta)$$
with a twisted group law: $(\al,x,l).(\beta,y,m)=(\al.\beta.m(x),x+y,l.m)$  \cite{Mu1}. 

Consider the $\comx$--vector space 
$$V(\delta)\,=\,\{f:\f{\Z}{\delta_1\Z}\times...\times \f{\Z}{\delta_g\Z}\lrar \comx \}$$
and the action of $(\al,x,l)\in Heis(\delta)$ on $f\in V(\delta)$ is given as :
$$(\al,x,l).f(y)=\al l(y).f(x+y).$$

Then we have
\begin{theorem}\label{th.-mu1} 
The $\comx$--vector space $V(\delta)$
is of dimension equal to $\delta_1.\delta_2...\delta_g$ and is the unique irreducible representation of 
the 
Heisenberg group $Heis(\delta)$ such that $\al \in \comx^*$ acts by its natural character. 
\end{theorem}
\begin{proof} See \cite[Proposition 1]{Mu1}.
\end{proof}

\noindent
\textbf{Definition}: If $W$ is a representation of the Heisenberg group $Heis(\delta)$ such that
$\al\in \comx^*$ acts as multiplication by $\al^l$, then we say that $W$ is a 
$Heis(\delta)$--module of \textit{weight} $l$.

We have the following result on higher weight $Heis(\delta)$--modules.

\begin{proposition}\label{pr.-iy1}
The set of irreducible representations of the Heisenberg group $Heis(\delta)$ of weight $l$
is in bijection with the set of characters on the subgroup of $l$--torsion elements,
$$K(\delta)_l \subset K(\delta).$$
Moreover the dimension of any such representation is 
$$\f{\delta_1...\delta_g}{(l,\delta_1)...(l,\delta_g)}.$$
If $\chi$ is a character on $K(\delta)_l$ and $W_\chi$ is the corresponding irreducible representation then 
$W_\chi\otimes \chi^{-1}$ is identified with the $Heis(\f{\delta}{l})$--representation 
$V(\f{\delta}{l})$ of weight $1$.
Here $\f{\delta}{l}=(\f{\delta_1}{(l,\delta_1)},...,\f{\delta_g}{(l,\delta_g)})$ and $(l,\delta_i)$ denotes the greatest common divisor of $l$ and $\delta_i$. 
\end{proposition}

\begin{proof}
See \cite[Proposition 3.2]{Iy1} when $l=2$ and \cite[Proposition 5.1]{Iy2} when $l\,>\,2$.
\end{proof}
 
\subsection{ $H^0(\cSU_C(r),\Theta_C)$ as a $Heis(\delta)$--module of weight $1$}

Given a nonsingular projective curve $C$ of genus $g$ and integers $r,d$, the moduli space of semi--stable 
vector bundles of rank $r$ and degree $d$ is denoted by $\cU_C(r,d)$. The moduli space of semi--stable bundles on $C$ of rank $r$ and trivial 
determinant is denoted by $\cSU_C(r)$ and the ample polarization on it by $\Theta_C$ \cite{Dr-Na}.
 The Jacobian $J_C^n$ 
parametrises degree $n$ line bundles on $C$, upto isomorphisms. 

Notice that the subgroup $(J_C)_r$ of $r$-torsion points on $J_C$, acts on the moduli space $\cSU_C(r)$ 
$$E\mapsto E\otimes l, \m{ for } l\in Pic^0(C)_r=J(C)_r$$
 and it leaves 
$\Theta_C$ invariant \cite[p.178]{BNR}. 

Consider the commutative diagram (I):
\begin{eqnarray*}
\Theta^k_C & \sta{\phi}{\simeq} & \Theta^k_C\\
\downarrow \,& &\,\downarrow \\
\cSU_C(r)  & \sta{\otimes l_r}{\lrar} &\cSU_C(r)
\end{eqnarray*}
Here $l_r$ is the line bundle corresponding to a $r$-torsion point on $J(C)$.

Consider the group 
$$G_k(\Theta_C)\,=\,\{(l_r,\phi): \Theta^k_C \sta{\phi}{\simeq} (\otimes l_r)^*\Theta^k_C\}.$$
Then there is an exact sequence:
$$1\lrar \comx^*\lrar G_k(\Theta_C)\lrar J(C)_r \lrar 0$$ 
which is a central extension.

We recall the constructions in \cite{BNR} which leads to a description of the
 vector space $H^0(\cSU_C(r),\Theta_C)$.

Firstly, the moduli space $\cU_C(r,d)$ is described as follows.

\begin{theorem}\label{th.-bnr}
There is a $r$--sheeted (ramified) covering $\pi:C'\lrar C$ with $C'$ nonsingular and irreducible such 
that the rational map $\pi_*:J_{C'}^\beta\lrar \cU_C(r,d)$ is dominant. The indeterminacy
locus of $\pi_*$ is of codimension at least $2$ and $\beta\,=\,d\,-\,\m{deg}\pi_*(\cO_{C'})$.
\end{theorem} 

\begin{proof} 
See \cite[Theorem 1]{BNR}.
\end{proof}

Let $\sigma=(\m{det}\pi_*\cO_{C'})^{-1}$ be the line bundle and consider the 
norm map 
$$
Nm: J_{C'}^{\m{deg}\sigma}\lrar J_C^{\m{deg}\sigma}.
$$
Let $P'= Nm^{-1}\sigma$ be the variety associated to the ramified covering $\pi: C'\lrar C$. We denote
$g= \m{ genus of }C$ and $g'=\m{ genus of }C'$. Then
there is a commutative diagram (I) (\cite[Proposition 5.7, p. 178]{BNR}) :

\begin{eqnarray*}
P'\times J^{g-1}_C & \sta{is}{\lrar} & J^{g'-1}_{C'}\\
\downarrow \pi_{*,is}\,& &\,\,\,\,\downarrow \pi_*\\
\cSU_C(r)\times  J^{g-1}_C & \sta{is_U}{\lrar} &\cU(r,r(g-1))
\end{eqnarray*}
and satisfying:

P.1. the morphism $is$ is an isogeny of degree $r^{2g}$ and $is_U$ is the map given by tensor product. Further, $(is_U)^*\Theta_U \simeq p_1^*\Theta_C \otimes p_2^*\Theta_J$, for the natural projections $p_i$.

P.2. $\pi_*$ induces a dominant (generically finite) rational map
$$\pi_{*,is}:P'\lrar \cSU_C(r).$$ 
The indeterminacy locus of $\pi_{*,is}$ is of codimension at least $2$.

P.3. $\Theta_{P'} = (\pi_{*,is})^*(\Theta_C)$ is a primitive line bundle (i.e., not a power of another line bundle) and is of type $\delta=(1,1,...,1,r,r,...,r)$. Here $r$ occurs $g$--times.

P.4. The subgroup $(J_C)_r$ of $r$--torsion points of $J_C$ acts on $\cSU_C(r)$ and leaves
the line bundle $\Theta_C$ invariant. 
There is a $\cG(\Theta_{P'}) $--action on the
sections of $\Theta_{P'}$ such that
the pullback map $H^0(\cSU_C(r),\Theta_C)\lrar H^0(P',\Theta_{P'})$ is equivariant for
this group and the pullback map is an isomorphism.

 Consider the commutative diagram (II):
\begin{eqnarray*}
P' \,\,\,\,\,\,\,\,\,\, & \sta{\otimes l_r}{\lrar} & \,\,\,P'\\
\downarrow \pi_{*,is}\,& &\,\,\,\,\downarrow \pi_{*,is}\\
\cSU_C(r)  & \sta{\otimes l_r}{\lrar} &\cSU_C(r)
\end{eqnarray*}
Here $l_r \in \m{Pic}^0(C)_r=J(C)_r$.

\begin{remark}\label{re.-is}
 Notice that \textrm{P.3} and (II) imply that $G_1(\Theta_C)\simeq \cG(\Theta_{P'})$.
Indeed, the map is given by $(l_r,\phi)\mapsto (l_r,(\pi_{*, is})^*\phi)$ which is injective and hence an isomorphism.
Further, this implies that the Weil pairing (given by the commutator map) on $J(C)_r$, corresponding to the extension
$$1\lrar \comx^*\lrar G_1(\Theta_C)\lrar J(C)_r\lrar 0$$
is nondegenerate. Also, $G_1(\Theta_C)$ acts on $H^0(SU_C(r),\Theta_C)$ with weight $1$ and is an irreducible representation.
\end{remark}

\begin{remark}\label{re.-is3}
The above mentioned remark can be extended to the following case:
consider the moduli space $\cSU_C(r,\eta)$ of semi-stable bundles with fixed determinant $\eta$.
Now $l_r\in J(C)_r$ acts on $\cSU_C(r,\eta)$ as $E\mapsto E\otimes l_r$. 
Since $\m{Pic }\cSU_C(r,\eta)=\Z.\Theta_C$ (\cite{Dr-Na}) any point $l_r$ of $J(C)_r$ corresponds to a finite order automorphism of $\cSU_C(r,\eta)$, we have $\Theta_C\simeq (\otimes l_r)^*\Theta_C$.
As earlier we can form the group of automorphisms $G_1(\Theta_C)$ of $\Theta_C$.
Further, there is a Weil form on $J(C)_r$, given by the commutator map associated to the extension,
$$1\lrar \comx^*\lrar G_1(\Theta_C)\lrar J(C)_r\lrar 0.$$
This form is nondegenerate since $\Theta_C$ is primitive.
 In other words, $G_1(\Theta_C)$ can be identified with the standard Heisenberg group $Heis(\delta)$, where
$\delta=(r,...,r)$ and $r$ occurs $g$-times.  
\end{remark}

\begin{remark}\label{re.-is2}
Since there is a (surjective) homomorphism
$$G_1(\Theta_C)\lrar G_k(\Theta_C),\,(x,\phi) \mapsto (x,\phi^{\otimes k}) $$
we see that $G_1(\Theta_C)$ acts on $H^0(SU_C(r,\eta),\Theta^k_C)$ and $\al\in \comx^*$ acts as $\al\mapsto \al^k$, i.e., with weight $k$. 
\end{remark}

\section{Parabolic case}

Suppose $C$ is a nonsingular projective connected curve of genus $g$ and $E$ is a vector bundle on $C$. Fix a 
parabolic data $\Delta$:

$$ S=\{x_i: 1\leq i\leq n\}\subset C \m{ is a finite set of }n \m{ distinct points},$$
fix a positive integer $m$ and for each $x\in S$ associate a sequence of integers
$$0<a_1(x)< a_2(x)<...< a_{l_x+1}(x)<m$$
called \textit{weights} $a(x)=(a_1(x),...,a_{l_x+1}(x))$. 
The weights $a(x)$ have \textit{multiplicities} $n(x)=(n_1(x),n_2(x),...,n_{l_x+1}(x))$
associated to a flag of the fibre $E_x$
$$E(x)=F_0(E_x)\supset F_1(E_x)\supset ... F_{l_x}(E_x)\supset F_{l_x+1}(E_x)=0$$
such that $n_j(x)=\m{dim}(\f{F_{j-1}(E_x)}{F_j(E_x)})$.

Consider the moduli space $\cSU_C(r,\Delta)$ of vector bundles of rank $r$ and trivial 
determinant and which are semi--stable with respect to the parabolic data $\Delta$. Then 
$\cSU_C(r,\Delta)$ is a projective variety (\cite{Me-Se}).
There is a \textit{parabolic theta line bundle} $\Theta_\Delta$ on $\cSU_C(r,\Delta)$ which is ample 
(\cite[Theorem 1.(A)]{Na-Ra}).  

We briefly recall the constructions (see also \cite{Su}):

Consider the Quot-scheme $\cQ$ of coherent sheaves of rank $r$ and degree $0$ over $C$ and trivial determinant, which are 
quotients of $\cO^{P(N)}(-N)$, with a fixed Hilbert polynomial $P$. Here $N$ is chosen large enough
so that every $\Delta$--parabolic semi--stable vector bundle with Hilbert
polynomial $P$ occurs as a point in $\cQ$.

Thus on $C\times\cQ$, there is a universal sheaf $\cF$, flat over $\cQ$ and denote the restriction
on $x\times \cQ$ by $\cF_x$, for $x\in S$. Let 
$$Flag_{n(x)}(\cF_x)\lrar \cQ$$ 
be the relative Flag scheme of type $n(x)$.  Consider the fibre product
$$\cR\,=\, \times_{x\in S}Flag_{n(x)}(\cF_x) \sta{pr}{\lrar} \cQ. $$
Let $\cR^{ss}\subset \cR$ denote the open subscheme of $\cR$ whose points correspond to 
$\Delta$--parabolic semi--stable bundles with trivial determinant. 
The pullback of $\cF\lrar C\times \cQ$, under $Id\times pr$, to $C\times \cR^{ss}$ is still denoted by 
$\cF$.

Denote the quotients 
$$\cQ_{x,i}\,=\, \f{\cF_{x}}{F_i(\cF_x)}.$$
The parabolic theta line bundle is defined as
\begin{equation}\label{eq.-t1}
\Theta_{\Delta}\,=\,(\m{det}R\pi_*(\cF))^m\otimes \bigotimes_{x\in S}((\m{det}\cF_x)^{m-a_{l_x+1}}\otimes \bigotimes_{i=1}^{l_x} 
(\m{det}\cQ_{x,i})^{a_{i+1}(x)-a_i(x)}).
\end{equation}

Here $\pi:C\times \cR^{ss}\lrar \cR^{ss}$ is the second projection and 
$$\m{det}R\pi_*(\cF)\,=\,(\m{det}\pi_*\cF)^{-1}\otimes \m{det}R^1\pi_*(\cF).$$

The variety $\cSU_C(r,\Delta)$ is the `good quotient' of $\cR^{ss}$ under the action of $SL(P(N))$.
The ample line bundle $\Theta_\Delta$ descends to an ample line bundle on $\cSU_C(r,\Delta)$ and is 
still denoted by $\Theta_\Delta$.

\begin{remark}\label{re.-p1}
Consider the open subscheme $\cQ^0\subset \cQ$ whose points correspond to semi--stable vector bundles
(in the usual sense). Then $SL(P(N))$ acts on $\cQ^0$ and there are rational dominant maps

\begin{equation}\label{eq.-p1}
q_1:\cQ^0\lrar \cSU_C(r)
\end{equation}
\begin{equation}\label{eq.-p2}
q_2:\cSU_C(r,\Delta)\lrar \cSU_C(r)
\end{equation}
\end{remark}
\begin{remark}\label{re.-p2}
 Further, the ample line bundle $\m{det}R\pi_*(\cF)$ on $\cQ^0$ descends to the theta line bundle
$\Theta_C$ on $\cSU_C(r)$. If $m=1$, we write $\cSU_C(r,\Delta)=\cSU_C(r)$.
\end{remark}

\subsection{The space $H^0(\cSU_C(r,\Delta),\Theta_\Delta)$ is a $G_1(\Theta_C)$--module}

Firstly, notice that the group $J(C)_r$ acts on the moduli space $\cSU_C(r,\Delta)$:
$$E\mapsto E\otimes l_r$$
for a line bundle $l_r\in \m{Pic}^0C=J(C)_r$.

In fact, there is a commutative diagram (III):
\begin{eqnarray*}
\cSU_C(r,\Delta) & \sta{\otimes l_r}{\lrar} & \cSU_C(r,\Delta)\\
\downarrow q_2\,& &\,\,\,\,\downarrow q_2\\
\cSU_C(r)  & \sta{\otimes l_r}{\lrar} &\cSU_C(r)
\end{eqnarray*}
 
\begin{lemma}\label{le.-t1}
Suppose the indeterminacy of the map $q_2$ is of codimension at least $2$. Then
the vector space $H^0(\cSU_C(r,\Delta),\Theta_\Delta)$ is a $G_1(\Theta_C)$--module of weight $m$.
\end{lemma}

\begin{proof}
Since the indeterminacy of $q_2$ is of codimension at least $2$, the pullback of $\Theta_C$ defines a line bundle on
$\cSU_C(r,\Delta)$. 
Further, it follows from \eqref{eq.-t1} that, $\Theta_\Delta = q^*\Theta^m_C \otimes M$, for some line bundle $M$ on $\cSU_C(r,\Delta)$ which is not a pullback from $\cSU_C(r)$.
Hence, given an element $(l_r,\phi)\in G_1(\Theta_C)$,
there is an isomorphism

$$\tilde\phi=q^*\phi^{\otimes m}\otimes Id: \Theta_\Delta \simeq (\otimes l_r)^*\Theta_\Delta$$
over $\cSU_C(r,\Delta)$.

This gives an action of $G_1(\Theta_C)$ on the space of sections $H^0(\cSU_C(r,\Delta),\Theta_\Delta)$.
In particular, the scalars act as $\al\mapsto \al^m$. This proves our assertion.

\end{proof}

\begin{corollary}\label{co.-5}
The vector space $H^0(\cSU_C(r,\Delta),\Theta^k_\Delta)$ is a $G_1(\Theta_C)$--module of weight $km$.
\end{corollary}
\begin{proof}
Indeed, as shown in Lemma \ref{le.-t1}, $(l_r,\phi)\in G_1(\Theta_C)$ induces isomorphisms
$$ \Theta^k_\Delta \sta{\tilde\phi^{\otimes k}}{\simeq} (\otimes l_r)^* \Theta^k_\Delta $$
over $\cSU_C(r,\Delta)$. 
Thus $\al\in \comx^*$ acts on $H^0(\cSU_C(r,\Delta),\Theta^k_\Delta)$ as $\al\mapsto \al^{km}$.
\end{proof}

Suppose $\delta=(r,r,...,r)$ with $r$ occuring $g$ times.

\begin{lemma}\label{le.-iy3}
Given  a level $r$--structure on the Jacobian $J(C)$, there is an isotypical decomposition

$$H^0(\cSU_C(r,\Delta),\Theta^k_\Delta)\,\simeq\, \bigoplus_{\chi\in \widehat{K(\delta)_{km}}} \, n_\chi.W_\chi$$ 

where $W_\chi$ is an irreducible representation of $Heis(\delta)$ of weight $km$. Moreover,
$W_\chi\otimes \chi^{-1}$ is identified with the $Heis(\f{\delta}{km})$--representation 
$V(\f{\delta}{km})$ of weight $1$.
\end{lemma}

\begin{proof}  
A level $r$--structure $h:J(C)_{r}\simeq K(\delta)$ is induced by an isomorphism 
$$G_1(\Theta_C)\simeq Heis(\delta).$$
This is true by Remark \ref{re.-is} and the arguments in \cite[p.318]{Mu1}): consider the subgroups 
$h^{-1}(K_1(\delta)), h^{-1}(\widehat{K_1(\delta)})\subset J(C)_r$. Consider their lifts which are level 
subgroups 
$$\tilde{K_1(\delta)},\tilde{\widehat{K_1(\delta)}}\subset G_1(\Theta_C).$$
 Construct
$f:G_1(\Theta_C)\lrar Heis(\delta)$ by mapping $\tilde{K_1(\delta)}$ onto the subgroup $\{(1,x,0): x\in K_1(\delta)\}$ and $\tilde{\widehat{K_1(\delta)}}$ onto the subgroup 
$\{(1,0,l): l\in \widehat{K_1(\delta)}\}$.
Now extend multiplicatively to obtain an isomorphism $G_1(\Theta_C)\simeq Heis(\delta)$.

Hence, by Remark \ref{re.-is2} and Corollary \ref{co.-5}, $H^0(SU_C(r,\Delta),\Theta^k_\Delta)$ is now a $Heis(\delta)$--module of weight $km$. 
By Proposition \ref{pr.-iy1},
there is an isotypical decomposition as asserted.

\end{proof}

\begin{definition}
An isomorphism $G_1(\Theta_C)\simeq Heis(\delta)$ is called a generalized theta structure.
\end{definition}

\section{A decomposition of the Verlinde bundles of higher level}

\subsection{The Verlinde bundles of level $km$}

Fix a parabolic data $\Delta$ as in the previous section and satisfying the hypothesis in Lemma \ref{le.-t1}.

Consider a smooth projective family of curves with $n$--marked points
\begin{equation}\label{eq.-1}
\pi:\cC\lrar T
\end{equation}
of genus $g\,>\,0$ and suppose $T$ is nonsingular.

\begin{remark}\label{re.-c}
We may assume that $T$ is the moduli space of nonsingular projective connected $n$--marked curves of 
genus $g$, with suitable level structures, so that there is a universal curve over $T$. 

\end{remark}

We can associate to \eqref{eq.-1}, the following families:
\begin{equation}\label{eq.-2}
\pi_J: \cJ\lrar T
\end{equation}
is the family of Jacobian varieties of dimension $g$,
\begin{equation}\label{eq.-3}
\pi_r:\cSU(r)\lrar T
\end{equation}
is the family of moduli spaces of semi--stable vector bundles of rank $r$ and trivial determinant and
\begin{equation}\label{eq.-4}
\pi_S:\cSU(r,\Delta)\lrar T
\end{equation}
is the family of moduli spaces $\cSU_t(r,\Delta)$ of $\Delta$--parabolic semi--stable vector bundles on 
$\cC_t$ of rank $r$ and trivial determinant.
 
There is a line bundle $\Theta_\Delta$ (resp. $\Theta$) on $\cSU(r,\Delta)$ (resp. $\cSU(r)$) such that 
$\Theta_\Delta$ restricts on any 
fibre $\cSU_t(r,\Delta)$ (resp. $\cSU_t(r)$) 
to the \textit{parabolic theta bundle} $\Theta_{\Delta,t}$ (resp. $\Theta_t$) \cite{Dr-Na}, \cite{Na-Ra}.

\noindent
\textbf{Definition}: The vector bundles 
$$\cV_{r,km}\,=\,\pi_{S\,*}(\Theta^k_\Delta)$$
are called as the \textit{Verlinde bundles} of level $km$, for $k\,>\,0$.

\subsection{A decomposition of the Verlinde bundles}

We denote $$\gamma_{r,km}\,=\, \f{r^g}{(km,r)^g}.\sum_{\chi\in \widehat{K(\delta)_{km}}}n_\chi\,=\,\m{rank }\cV_{r,km}.$$

Consider the group scheme $\cJ_r\lrar T$ which is the kernel of the homomorphism
$$\cJ\lrar \cJ$$ given by multiplication by $r$ on $\cJ$. 
There is an exact sequence
$$1\lrar \G_{m,T} \lrar \cG_1(\Theta)\lrar \cJ_r\lrar 0$$
where $\cG_1(\Theta)$ represents the functor defining the automorphisms of $\Theta$ over the sections of $\cJ_r$ (see also \cite[p.76]{Mu2}, for similar constructions).
 
\begin{proposition}\label{pr.-1} 
Given a $t_0\in T$, there is an \'etale open cover $U\lrar T$ of $t_0$, such that 
$$ \cV_{r,km}\,\simeq \,\bigoplus_{\chi\in \widehat{K(\delta)_{km}}}  W_\chi \otimes F_\chi $$
over $U$ and for some vector bundles $F_\chi$ on $U$.
\end{proposition}
\begin{proof}
Suppose $T$ is the moduli space of nonsingular $n$--marked curves with level $r$--structure.
Given a $t_0\in T$, a level $r$--structure can be lifted locally on $T$ to a generalized theta structure, say 
over an open \'etale cover $U\lrar T$, i.e., the group scheme $\cG_1(\Theta)$ trivializes over $U$ and is identified with $Heis(\delta)\times U$.
Hence $Heis(\delta)\times U$ acts on the Verlinde bundle $\cV_{r,km}$ with weight $k$.

Now the proof is, by using Lemma \ref{le.-iy3} and the arguments in  \cite[Proposition 2, p.80]{Mu2}:
Since the subgroup $K(\delta)_{km}$ is represented over $T$, there is a vector bundle decomposition
 $$\cV_{r,km}\simeq\,\bigoplus_{\chi\in \widehat{K(\delta)_{km}}} \cW_\chi$$
where $\cW_\chi$ is a subbundle and is acted by the character $\chi$.

Over $U$, we know that $\cW_\chi$ is acted upon by $Heis(\f{\delta}{km})$ and hence
$$\cW_\chi\simeq W_\chi \otimes F_\chi$$
where $W_\chi$ is defined in section 2. and for some vector bundle $F_\chi$ of rank $n_\chi$. 

This gives the required isomorphism
\begin{equation}\label{eq.-t}
\cV_{r,km}\sta{\simeq}{\lrar} \bigoplus _{\chi\in \widehat{K(\delta)_{km}}} W_\chi \otimes F_\chi
\end{equation}
over $U$.

\end{proof}

\begin{corollary}\label{co.-2} 
The Chern character of the Verlinde bundle $\cV_{r,1}$ is written as
$$ ch(\cV_{r,1})\,=\,\gamma_{r,1}.ch(L_S)\,\in\,CH^*(T)_\Q$$
for some line bundle $L_S$ on $T$.
\end{corollary}
\begin{proof}
In the rational Grothendieck group $K^0(T)_\Q$, 
$$\cV_{r,1}\simeq L_S^{\oplus \gamma_{r,1}}$$
where $L_S=F_0$ is a line bundle.
This gives the assertion on the Chern characters in the rational Chow groups $CH^*(T)_\Q$.

\end{proof}

\begin{remark}
Since the moduli stack $\cM_g$ of curves has $\m{Pic}\cM_g =\Z.\lambda$ (\cite{Ar-Co}), where $\lambda$ is the first Chern class of the Hodge bundle $\pi_*\omega_{\cC/T}$, it follows that
$$L_S=l.\lambda \in CH^*(T)_Q, \m{ for some } l\in \Q,$$
(we may assume $T\lrar\cM_g$).
In particular, $$ch(\cV_{r,1}) = \gamma_{r,1}.ch(l.\lambda) \in CH^*(T)_\Q.$$

\end{remark}

\section{A remark on the multiplicities of the isotypical components}

In this section, we indicate how the multiplicity $n_\chi$ of the representation $W_\chi$ which occurs 
in $H^0(\cSU_C(r,\eta),\Theta^k_C)$  (here $n_\chi$ are as defined in Lemma \ref{le.-iy3}), can be computed. This was mentioned to us by A. Beauville.

Let $K\subset J(C)_r$ be any subgroup isomorphic to $\mu_s\times \mu_s$, $s\leq r$.
Consider the moduli space $M_{\f{SL(r)}{\mu_s}}$ of principal semistable $\f{SL(r)}{\mu_s}$ bundles. 

\subsection{The multiplicities when $K=J(C)_r$}\label{equal}

In this case we obtain the moduli space $M_{PGL(r)}$ of principal semi-stable $PGL(r)$-bundles.
Further, fix a point $p\in C$ and denote $L=\cO_C(d.p)$. Then we have \cite{Be1},
$$M_{PGL(r)}=\amalg_{0\leq d<r}M^d_{PGL(r)}$$
and $$M^d_{PGL(r)}=\f{\cSU_C(r,L)}{J(C)_r}.$$

Suppose $\Theta'$ denotes the primitive line bundle on $M^d_{PGL(r)}$ (i.e., the first power of the determinant line bundle which descends to the quotient) and 
$$\gamma^d_{r,k}= \m{dim} H^0(\cSU_C(r,L),\Theta^k_C).$$

By Remark \ref{re.-is2}, we can write
$$
\gamma^d_{r,k} = \sum_{\chi\in \widehat{J(C)_r}} n^d_{\chi}.\m{dim}W_\chi.
$$
Here $n^d_\chi$ is the multiplicity of $W_\chi$ which occurs in $H^0(\cSU_C(r,L),\Theta_C^k)$.

\begin{lemma}\label{le.-m}
Suppose $k$ is a multiple of $r$ and $r$ is odd
or if $k$ is a multiple of $2r$ and $r$ is even.
Then $$n_{\chi^{triv}} =  \m{ dim } H^0(M^d_{PGL(r)},\Theta')$$
and $$n_\chi = \f{\gamma^d_{r,k} - n_{\chi^{triv}}}{r^{2g}-1}, \,\chi\neq \chi^{triv}.$$  
\end{lemma}
\begin{proof}
In \cite{Be-La-So}, it is shown that 
$\Theta^k_C$ descends down to the quotient $M^d_{PGL(r)}$ if $k=l.r$ and $r$ is odd
or if $k=l.2r$ and $r$ is even.
We note that the $J(C)_r$-invariant sections of $H^0(\cSU_C(r,L),\Theta_C^k)$ is the isotypical component $n^d_{\chi^{triv}}.W_{\chi^{triv}}$ which is precisely the pullback of the space of sections 
$H^0(M^d_{PGL(r)},\Theta'^l)$. Further, by Proposition \ref{pr.-iy1}, it follows that $\m{dim}W_\chi=1$, for any $\chi \in \widehat{J(C)_r}$.
The assertion now follows from the equality
$$\gamma^d_{r,k}= \sum_{\chi\in \widehat{J(C)_r}}n^d_\chi.\m{dim}W_\chi$$
and noting that $n^d_\chi$ is constant, for any $\chi\neq \chi^{triv}$.
\end{proof}

\begin{remark}
In \cite[Proposition 3.4]{Be1}, when $r$ is a prime, $\m{dim} H^0(M^d_{PGL(r)},\Theta'^l)$ is computed. 
Hence we get an explicit formula for the multiplicities $n_{\chi^{triv}}$ and
$n_\chi$ in this case. 

\end{remark}

\subsection{The multiplicties when $K\subset J(C)_r$}

For a subgroup $K=\mu_s\times \mu_s\subset J(C)_r$, $s<r$, we consider the 
intermediate quotients
$$
\cSU_C(r,L) \lrar \f{\cSU_C(r,L)}{K} \lrar M^d_{PGL(r)}.
$$

As in \S\ref{equal}, the disjoint union
$$
M_{\f{SL(r)}{\mu_s}}:=\amalg_{0\leq d<r} \f{\cSU_C(r,L)}{K}
$$ 
is the moduli space  of principal semi-stable 
$\f{SL(r)}{\mu_s}$- bundles.

\begin{lemma}\label{le.-descend}
Given any integer $k$, there is a subgroup $K=\mu_s\times \mu_s\subset J(C)_r$, $s\leq r$, such that $\Theta^{k}_C$ 
descends down to the variety
$\f{\cSU_C(r,L)}{K}$ as a power of a primitive line bundle 
$\Theta'_K$.
\end{lemma}
\begin{proof}
Notice that the degeneracy of the Weil form on $J(C)_r$ associated to the exact sequence
$$1\lrar \comx^*\lrar G_1(\Theta^k_C)\lrar J(C)_r\lrar 0$$
is a subgroup $K=\mu_s\times \mu_s\subset J(C)_r$ for some $s\leq r$.
Hence there is a lift of $K$ in $G_1(\Theta^k_C)$ over $K$ which forms a descent data for the line bundle $\Theta^k_C$.
\end{proof}

As in \S \ref{equal}, we denote for $L=\cO(d.p)$ and $0\leq d < r$,
$$
\gamma^d_{r,k}= \m{dim} H^0(\cSU_C(r,L),\Theta^k_C)
$$
and $n^d_\chi$ is the multiplicity of $W_\chi$ which occurs in $H^0(\cSU_C(r,L),\Theta_C^k)$.

Then
$$
\gamma^d_{r,k}= \sum_{\chi\in \widehat{K}}n^d_\chi.\m{dim} W_\chi.
$$
Hence we write
\begin{eqnarray*}
\gamma_{r,k} & := & \sum_{0\leq d < r}\gamma^d_{r,k}\\ 
            & = &  \sum_{\chi\in \widehat{K}}n_\chi.W_\chi
\end{eqnarray*}
where $n_\chi = \sum_{0\leq d < r}n^d_\chi$.

\begin{lemma}
The multiplicities $n_\chi$, for any $\chi\in \widehat{K }$, can be computed.
\end{lemma}
\begin{proof}

By Lemma \ref{le.-descend}, there is an $s$, $0\leq s\leq r$ and $K=\mu_s\times \mu_s\subset J(C)_r$, such that
$\Theta^k_C$ descends to $\f{\cSU_C(r,L)}{K}$ as a power of a primitive line 
bundle $\Theta'_K$.
 By Remark \ref{re.-is3}, $H^0(\cSU_C(r,L),\Theta_C^k)$ is a $G_1(\Theta_C)$-module of weight $k$, for $L=\cO(d.p)$ and
$0\leq d< r$.
By conformal field theory (\cite{S-Y}), we know the vector space dimension
$$\sum_{0\leq d< r}\m{dim}H^0(M^d_{\f{SL(r)}{\mu_s}},{\Theta'}^l_K).$$ 

As shown in Lemma \ref{le.-m}, a similar argument gives the multiplicities
$n_\chi=\sum_{0\leq d<r} n^d_\chi$, for any $\chi\in \widehat{K}$. 

\end{proof}

\begin{remark}
If the dimensions of the individual  vector spaces $H^0(M^d_{\f{SL(r)}{\mu_s}},{\Theta'}^l_K)$ are known then we would be able to compute the individual multiplicties $n^d_\chi$.
\end{remark}

\subsection{A remark on the multiplicities $n_\chi$}

Let $\gamma_{r,k}=\m{dim} H^0(\cSU_C(r),\Theta_C^k)$.
Then by the Verlinde formula (\cite{Be-La},\cite{Fa2}), we have

$$\gamma_{r,k}\,=\,(\f{r}{r+k})^g.\sum_{\sta{S\sqcup R=[1,r+k]}{|S|=r}}\prod_{\sta{s\in S}{z\in R}}|2.\m{sin}\,\pi\f{s-z}{r+k}|^{g-1}.$$

Also, by Remark \ref{re.-is2}, we can write
\begin{eqnarray*}
\gamma_{r,k} & = & \sum_{\chi \in\widehat{(J(C)_r)_k}} n_\chi.\m{dim}W_\chi \\
             &= & \f{r^g}{(r,k)^g}.\sum_{\chi\in \widehat{(J(C)_r)_k}} n_\chi, \m{ by Proposition \ref{pr.-iy1}}. \\
\end{eqnarray*}

Comparing the above two expressions, we get

$$\sum_{\chi\in \widehat{(J(C)_r)_k}}n_\chi \, = \, \f{(r,k)^g}{(r+k)^g}. \sum_{\sta{S\sqcup R=[1,r+k]}{|S|=r}}\prod_{\sta{s\in S}{z\in R}}|2.\m{sin}\,\pi\f{s-z}{r+k}|^{g-1}.$$
(See also \cite{Za}, for the various aspects of the Verlinde formula.)


\end{document}